\documentclass[12pt]{amsart}

\usepackage{amsfonts, amstext, amsmath, amsthm, amscd, amssymb, graphicx, epstopdf}
\usepackage{epsfig, graphics, psfrag, enumerate}
\usepackage{color}
\usepackage{verbatim}
\let\oldmarginpar\marginpar
\renewcommand\marginpar[1]{\oldmarginpar[\raggedleft\footnotesize #1]%
{\raggedright\footnotesize #1}}

\newtheorem{theorem}{Theorem}[section]

\newtheorem{lemma}[theorem]{Lemma}
\newtheorem{corollary}[theorem]{Corollary}
\newtheorem{proposition}[theorem]{Proposition}

\newtheorem{question}[theorem]{Question}

\newtheorem{define}[theorem]{Definition}

\theoremstyle{definition}
\newtheorem{remark}[theorem]{Remark}
\newtheorem{example}[theorem]{Example}

\newcommand{\ZZ}{{\mathbb{Z}}}
\newcommand{\NN}{{\mathbb{N}}}

\newcommand{\QQ}{{\mathbb{Q}}}

\newcommand{\bdy}{{\partial}}

\newcommand{\abs}[1]{{\left\vert #1 \right\vert}}

\newcommand\no[1]{}

\newtheorem*{namedtheorem}{\theoremname}
\newcommand{\theoremname}{testing}

\def\la{\langle}
\def\ra{\rangle}

\def\be { \begin{equation} }
\def\ee { \end{equation} }

\begin{document}

\title[]{Normal and Jones surfaces of knots}

\author[]{Efstratia Kalfagianni}
\address{Department of Mathematics, Michigan State University, E. Lansing, MI, 48824}
\email{kalfagia@math.msu.edu}
\author[]{Christine Ruey Shan Lee}
\address{Department of Mathematics, University of Texas at Austin, Austin, TX 78712 }
\email{clee@math.utexas.edu}
\bigskip

\bigskip

\begin{abstract} We describe a normal surface algorithm that decides whether a knot, with known degree of the colored Jones polynomial, satisfies the Strong Slope Conjecture. 
We also discuss possible simplifications of our
algorithm and state related open questions. We establish  a relation between the Jones period of a knot and the
number of sheets of the surfaces that satisfy the Strong Slope Conjecture (Jones surfaces).  We also present numerical and experimental evidence supporting 
a stronger such relation which we state as an open question.

\end{abstract}

\bigskip

\bigskip
\thanks {\today}
\thanks{Kalfagianni is supported in part by NSF grants DMS--1404754 and  DMS-1708249}
\thanks{Lee is supported in part by NSF grant DMS--1502860}

\maketitle

\section{Introduction}

The Strong Slope Conjecture,   stated by the first named author and Tran in  \cite{Effie-Anh-slope}, refines the Slope Conjecture of Garoufalidis \cite{ga-slope}. It  has made explicit a close relationship between the degrees of the colored Jones polynomial and essential surfaces in the knot complement.  In particular, it implies that the linear terms in the degrees of the colored Jones polynomial detect the unknot. The conjecture predicts that the asymptotics of the degrees determine the boundary slopes and the ratios of the Euler characteristic to the number of  sheets  of essential surfaces in the knot complement. Such surfaces are called
\emph{Jones surfaces} (see Section \ref{sec:conj}). Not much is known about the nature of these Jones surfaces, and it is unclear how they are distinguished from other essential surfaces of the knot
complement. 

Our purpose in this paper is two-fold: 
On one hand we are interested in the information about the topology of Jones surfaces encoded in the {\emph{period}} of the degree of the colored Jones polynomial.
On the  other hand  we are interested in the question of how Jones surfaces behave with respect to {\emph{normal surface}} theory in the knot complement.

We show that there is a relation between 
the number of sheets of a Jones surface, the Euler characteristic and the period of the knot (see Proposition \ref{divide}).
Then,  we present numerical evidence suggesting that  the number of sheets of a Jones surface should divide the period of the knot. See Examples \ref{ex:1}-\ref{ex:pretzel} and Question \ref{lessthanp}. 
We also  examine  the Jones surfaces of knots from the viewpoint of normal surface theory in the knot complement. 
The question that we are concerned with is the following: 
If a knot satisfies the Strong Slope Conjecture, what are the simplest Jones surfaces, in the sense of normal Haken sum decompositions \cite{Haken}?
 In particular can we find Jones surfaces that are {\emph{fundamental}} in the sense of Haken? 
As a result  of our analysis, and using Proposition \ref{divide}, we show that there is an algorithm  to decide whether a  knot, with given asymptotics of the degree of its colored Jones polynomial,  satisfies the Strong Slope Conjecture. 

To state our result, for a knot  $K \subset S^3$ let  $J_{K}(n)$ denote the $n$-th colored Jones polynomial of $K$ and let
$d_+[J_{K}(n)]$ and  $d_-[J_{K}(n)]$ denote  the maximal and minimal degree of $J_{K}(n)$, respectively. See Section 2 for definitions and details.
Our main result is a slightly stronger version  (see Theorem \ref{mainjones }) of the following
\begin{theorem} \label{thm:main} Given a knot $K$ with known $d_+[J_{K}(n)]$ and  $d_-[J_{K}(n)]$, 
there is a normal surface theory algorithm that decides whether $K$ satisfies the Strong Slope Conjecture.
\end{theorem}

The paper is organized as follows:
In Section 2 we state the Strong Slope Conjecture and briefly survey the cases where the conjecture has been proven. 
In Section 3, we discuss the relations between the Jones period and the number of sheets of Jones surfaces of knots.
In Section 3 we study Haken sum decompositions of Jones surfaces and we prove Theorem \ref{mainjones } which implies Theorem \ref{thm:main}. We also discuss possible simplifications of our algorithm and state related open questions.

We thank Josh Howie for useful comments and  bringing to our attention an oversight in the proof of Theorem 4.3 in an earlier version of the paper.

\section{Jones slopes and surfaces} \label{sec:conj}
\subsection{Definitions and statements} We recall the definition of the  colored Jones polynomial; for more details the reader is referred to \cite{Li}:
  We first recall the definition of the Chebyshev polynomials of the second kind.
   For $n \ge 0$, the polynomial $S_n(x)$ is defined recursively as follows:
\begin{equation}
\label{chev}
S_{n+2}(x)=xS_{n+1}(x)-S_{n}(x), \quad S_1(x)=x, \quad S_0(x)=1.
\end{equation}

Let $D$ be a diagram of a knot $K$. For an integer $m>0$, let $D^m$ denote the diagram obtained from $D$ 
by taking $m$ parallel copies of $K$. This is the $m$-cable of $D$ using the blackboard framing. If $m=1$ then $D^1=D$. 
Let $\la D^m \ra$ denote the Kauffman bracket of $D^m$. This is a Laurent polynomial over the integers in the variable $t^{-1/4}$, normalized so that $\la \text{unknot} \ra = -(t^{1/2}+t^{-1/2})$. Let $c=c(D)=c_+ + c_-$ denote the crossing number and $w=w(D)=c_+ - c_-$ denote the writhe of $D$. 

For $n>0$, we define 
$$J_K(n):=  ( (-1)^{n-1} t^{(n^2-1)/4} )^w (-1)^{n-1}  \la S_{n-1}(D)\ra, $$
where $S_{n-1}(D)$ is a linear combination of blackboard cables of $D$, obtained via equation \eqref{chev}, and the notation $\la S_{n-1}(D) \ra$ means extend the Kauffman bracket linearly. That is, for diagrams $D_1$ and $D_2$ and scalars $a_1$ and $a_2$, $$\la a_1 D_1 + a_2 D_2 \ra = a_1 \la D_1 \ra + a_2 \la D_2 \ra.$$

For a knot  $K \subset S^3$ let $d_+[J_{K}(n)]$ and  $d_-[J_{K}(n)]$ denote  the maximal and minimal degree of $J_{K}(n)$ in $t$, respectively.

Garoufalidis \cite{ga-quasi}  showed that
the degrees $ d_+[J_{K}(n)] $ and $ d_-[J_{K}(n)] $ are quadratic {\em quasi-polynomials}.
This means that, given a knot $K$, there is $n_K\in \NN$ such that  for all $n>n_K$ we have
 $$ \, d_+[J_{K}(n)] =  a_K(n) n^2 + b_K(n) n  + c_K(n), \ \ \    \, d_-[J_{K}(n)] =  a_K^{*}(n) n^2 + b_K^{*}(n) n  + c_K^{*}(n),$$
 where the coefficients are periodic functions from $\NN $ to $\QQ$ with finite integral period.

\begin{define} {\rm The least common multiple of the periods of all the coefficient functions is called the {\em Jones period} $p$ of $K$.}
\end{define}

For a sequence $\{x_n\}$, let $\{x_n\}'$ denote the set of its cluster points.

\begin{define}  \label{jslopes} {\rm  An element of the sets 
$$js_K:= \left\{ 4n^{-2}d_+[J_K(n)]  \right\}',  \quad
  js^*_K:= \left\{ 4n^{-2}d_-[J_K(n)] \right\}' $$
 is called a {\em Jones slope} of $K$.  
 Also  let
$$jx_K:= \left\{ 2n^{-1}\ell d_+[J_K(n)]  \right\}'=\left\{2b_K(n)\right\}',\ \ \   jx^{*}_K:= \left\{ 2n^{-1}\ell d_-[J_K(n)]  \right\}'=\left\{2b^{*}_K(n)\right\}',$$
where  $\ell d_+[J_K(n)]$  and  $\ell d_-[J_K(n)]$  denote the linear term of $d_+[J_K(n)]$ and $d_-[J_K(n)]$, respectively. }

 \end{define}
 
 Given a knot $K\subset S^3$, let
  $n(K)$ denote a tubular neighborhood of
$K$ and let $M_K:=\overline{ S^3\setminus n(K)}$ denote the exterior of
$K$. Let $\langle \mu, \lambda \rangle$ be the canonical
meridian--longitude basis of $H_1 (\bdy n(K))$.   A properly embedded surface $$(S, \bdy S) \subset (M_K,
\bdy n(K)), $$ is called \emph{essential} if it is $\pi_1$-injective and it is not a boundary parallel annulus.

An element $a/b \in
{\QQ}\cup \{ 1/0\}$ with $\gcd(a, b)=1$ is called a \emph{boundary slope} of $K$ if there is an essential surface $(S, \bdy S) \subset (M_K,
\bdy n(K))$, such that  $\bdy S$ represents $[a \mu + b \lambda] \in
H_1 (\bdy n(K))$.  Hatcher showed that every knot $K \subset S^3$
has finitely many boundary slopes \cite{hatcher}. The {\em Slope Conjecture} \cite[Conjecture 1]{ga-slope} asserts that the Jones slopes of any knot $K$ are 
 boundary slopes.
The {\em  Strong Slope Conjecture} \cite[Conjecture 1.6]{Effie-Anh-slope} asserts that the topology of the surfaces realizing these boundary slopes may be predicted
by  the linear terms of  $d_+[J_K(n)]$, $d_-[J_K(n)]$.
\vskip 0.09in

\noindent { {\bf Strong Slope Conjecture.}}{\emph { Given
a Jones slope of $K$, say
 $a/b\in js_K$, with $b>0$ and $\gcd(a, b)=1$, there is an essential surface $S\subset M_K$ with $\abs{\partial S}$  boundary components such that each component of $\partial S$ has slope $a/b$,  and
$$\frac{\chi(S)}{{\abs{\partial S} b}} \in jx_K.$$
Similarly,  given  $a^{*}/b^{*}\in js^{*}_K$, with $b^{*}>0$ and $\gcd(a^{*}, b^{*})=1$,  there is an essential surface $S^{*}\subset M_K$ with $\abs{\partial S^{*}}$  boundary components such that each component of $\partial S^{*}$ has slope $a^{*}/b^{*}$, and
$$-\frac{\chi(S^{*})}{{\abs{\partial S^{*}} b^{*}}} \in jx^{*}_K.$$}}

\begin{define}  \label{jchar}{\rm
 With the notation as above, a  {\em Jones surface} of  $K$ is an essential surface $S\subset M_K$ such that, either
\begin{itemize} 
\item $\partial S$ represents a Jones slope  $a/b\in js_K$, with $b>0$ and $\gcd(a, b)=1$, and  we have
$$ \frac{\chi(S)}{{\abs{\partial S} b}} \in jx_K; \ \ \  {\rm or}$$
\item$\partial S$ represents a Jones slope $a^{*}/b^{*}\in js^{*}_K$, with $b^{*}>0$ and $\gcd(a^{*}, b^{*})=1$, and we have  $$- \frac{\chi(S)}{\abs{\partial S} b^{*}} \in jx^{*}_K.$$

\end{itemize}

The number  $\abs{\partial S} b$ (or $\abs{\partial S}b^*$) is  called the {\em{number of sheets}} of the Jones surface.}
\end{define}

We note that the  Strong Slope Conjecture implies that the cluster points of the liner terms in the degree of the colored Jones polynomial alone, detect the trivial knot.
More specifically, we have:

\begin{theorem} If the Strong Slopes Conjecture holds for all knots. then the following is true:
A knot $K$ is the unknot if and only if $ js_K=\{1\}$.
\end{theorem}

\begin{proof} With the normalization of \cite{Effie-Anh-slope} if $K$ is the unknot we have $ js_K=\{1\}$.
On the other hand if  $js_K=\{1\}$, and the Strong Slopes Conjecture holds for $K$, then we have a Jones surface $S$  for $K$ with boundary slope $0$ and with $\chi(S)>0$.
Then $S$ must be a collection of discs which means that a Seifert surface for  $K$ is a disc  and thus $K$ is the unknot.

\end{proof}


\subsection{What is known}The Strong Slope Conjecture is known for the following knots.
\begin{itemize}

\item  Alternating knots \cite{ga-slope} and adequate knots \cite{FKP, FKP-guts}.
\item  Iterated   torus knots  \cite{Effie-Anh-slope}.
\item Families of 3-tangle pretzel knots \cite{LV} .
\item Knots with up to  9 crossings \cite{ga-slope, Howie17, Effie-Anh-slope}.
\item Graph knots \cite{MoTa}.
\item  An infinite family of arborescent non-Montesinos  knots \cite{Howie-Do}.
\item Near-alternating knots \cite{Lee17} constructed by taking Murasugi sums of an alternating diagram with a non-adequate diagram.
\item Knots obtained by iterated cabling and connect sums of knots from any of the above classes \cite{Effie-Anh-slope, MoTa}.
\end{itemize}
The Slope Conjecture is also known for a family of 2-fusion knots, which is a 2-parameter family $K(m_1, m_2)$ of closed 3-braids, obtained by the $(-1/m_1, -1/m_2)$ Dehn filling on a 3-component link $K$ \cite{GV}.


\section{Jones period and Jones surfaces}

We show that the Strong Slope Conjecture implies a relationship between the number of sheets of a Jones surface for a knot $K$, its Euler characteristic, and the Jones period.

\begin{proposition}\label{divide} Suppose that $K\subset S^3$ is  a knot of Jones period $p$.  Let $a/b \in  js_K\cup  js^{*}_K$ be a Jones slope and  let  $S$ be a corresponding Jones surface.
Then $b$ divides $p^2$ and $b|\partial S|$ divides $2p^2\chi(S)$. 

In particular, if $p=1$ then all the Jones slopes of $K$ are integral and for every  Jones surface we have $\displaystyle{\frac{2\chi(S)}{|\partial S|}}\in \ZZ$. 
\end{proposition}

\begin{proof}

Suppose, for notational simplicity,  that  $a/b \in  js_K$ and thus $S$ corresponds to the highest degree  $4d_+[J_{K}(n)] =  4a_K(n) n^2 + 4b_K(n) n  + 4c_K(n)$ for some fixed $n>n_K$ with respect to ${a}/{b}$.  The case
 $a/b \in js^{*}_K$ is completely analogous.
 
The claim that  $b$ divides $p^2$  is shown in \cite[Lemma 1.10]{ga-slope}.
By the above discussion we can assume that for some fixed $n >n_K$ with respect to ${a}/{b}$ and for every integer $m> 0$ we have
$$4a_K(n)=4a_K(n+mp)=\frac{a}{b} \ \ {\rm and} \ \  4b_K(n)= 4b_K(n+mp) = \frac{2\chi(S)}{{\abs{\partial S} b}},$$
while $4c_K(n+mp) = 4c_K(n).$
Furthermore we have $d_+[J_{K}(n)] $ and $d_+[J_{K}(n+mp)]$ are integers for all $m> 0$ as above.
If $\abs{\partial S} b=1$ then there is nothing to prove. Otherwise, set $m=1$ and consider
$4d_+[J_{K}(n+mp)] -4d_+[J_{K}(n)].$ 
We have
\begin{align*}
4d_+[J_{K}(n+p)] -4d_+[J_{K}(n)] &= \frac{2anp|\partial S| + ap^2|\partial S| + 2p\chi(S)}{b|\partial S|} \\ 
&= \frac{2anp}{b}+ \frac{ap^2}{b} + \frac{2p\chi(S)}{b|\partial S|}, 
\end{align*}
which must be an integer. Since $b$ divides $p^2$ , the term $ {ap^2}/{b}$ is an integer. We conclude that 
$$\frac{2anp}{b}+\frac{2p\chi(S)}{b|\partial S|}$$ 
is an integer. Multiplying the last quantity by $p$ also gives an integer; thus
$$\frac{2anp^2}{b} +\frac{2p^2\chi(S)}{b|\partial S|}$$ 
is an integer. But since $2{anp^2}/{b} $ is an integer, we have that ${2p^2\chi(S)}/{b|\partial S|}$ is an integer, and the conclusion follows that $b|\partial S|$ divides  $2p^2\chi(S)$.
\end{proof}

It turns out that for all knots where the Strong Slope Conjecture is known and the Jones period is calculated, for each Jones slope we can find a Jones surface where the number of sheets $b|\partial S|$  actually divides the Jones period. This leads us to give the following definition. 
\begin{define}\label{characteristic} {\rm
We call a Jones surface $S$ of a knot $K$ {\emph {characteristic}} if the number of sheets of $S$ divides the Jones period of $K$. }
\end{define}

\begin{example} \label{ex:1} An adequate knot has Jones period equal to 1, two  Jones  slopes and two corresponding
Jones surfaces each with a single boundary component \cite{FKP}.
By the proof of \cite[Theorem 3.9]{Effie-Anh-slope}, this property also holds for iterated cables of adequate knots. Thus in all the cases, we can find characteristic Jones surfaces.
Note, that for adequate knots the characteristic Jones surfaces are spanning surfaces that are often non-orientable. In these cases the orientable double cover is also a Jones surface
but it is no longer characteristic since it has two boundary components. 
\end{example}

\begin{table}[ht]

\begin{tabular}{|c|c|c|c|c|c|c|c|c|c|c|}
\hline
Knot & $js_K$ & $|\partial S|$ & $\chi(S)$ & $b|\partial S|$  & $js^*_K$ & $|\partial S^*|$ & $\chi(S^*)$ & $b^*|\partial S^*|$ & $p$\\ 
\hline
$8_{19}$ &\{12\} & 2 & 0 & 2 & \{0\} & 1 & -5 & 1 & 2\\  
\hline
$8_{20}$ &\{8/3\} & 1& -3 &3 & \{-10 \} & 1 & -4 & 1 & 3\\ 
\hline
$8_{21}$ & \{1\} & 2 & -4 & 2 & \{-12\} & 1 & -3 & 1 & 2 \\ 
\hline
$9_{42}$  & \{6\} & 2 & -2& 2 & \{-8\} & 1 & -5 & 1 & 2\\
\hline 
$9_{43}$  & \{32/3\} & 1 & -3 & 3 & \{-4\} & 1 & -5 & 1 & 3 \\ 
\hline
$9_{44}$  & \{14/3\} & 1 & -6 &3 & \{-10\} & 1 & -4 & 1  & 3\\
\hline 
$9_{45}$ & \{1\} & 2 & -4 & 2 & \{-14\} & 1 & -4 & 1 & 2 \\  
\hline
$9_{46}$ & \{2\} & 2 & -2 & 2 & \{-12\} & 1 & -5 & 1 & 2 \\ 
\hline 
$9_{47}$ & \{9\} & 2 & -4& 2 & \{-6\} & 1 & -4 & 1 & 2 \\ 
\hline 
$9_{48}$ & \{11\} & 2 & -6 & 2 &\{-4\} & 1 & -3 & 1 & 2 \\ 
\hline 
$9_{49}$ & \{15\} & 2 & -6 & 2 & \{0\}& 1 & -3 & 1 & 2  \\
\hline
\end{tabular} 
\vskip 0.1in
\caption{Eight and nine crossing non-alternating knots}\end{table}
\begin{example}For each non-alternating knot up to nine crossings, Table 1  gives the Jones period,  the Jones slopes, and the numbers of sheets of corresponding characteristic  Jones surfaces. 

The Jones slopes and Jones period in the table  are compiled from \cite{ga-slope}. The Jones surface data for all examples,   but $9_{47}$ and $9_{49}$, are obtained from \cite{Effie-Anh-slope}. 
The proof that the knots $9_{47}$ and $9_{49}$ satisfy the Strong Slope Conjecture was recently done by Howie \cite{Howie17}. 

\end{example}

\begin{example} 
 By \cite{Effie-Anh-slope}, the Jones slopes of 
a $(p, q)$-torus knot $K = T(p, q)$ are $pq$ and $0$,  with Jones surfaces an annulus and a minimum genus Seifert surface, respectively.
The Jones period of $K$ is 2 and thus both Jones surfaces are characteristic.
By the proof of \cite[Theorem 3.9]{Effie-Anh-slope}, this property also holds for iterated torus knots.
\end{example} 

\begin{example} \label{ex:pretzel} Consider the pretzel knot $K=P(1/r, 1/s, 1/t)$ where $r, s, t$ are odd, $r < 0$, and $s, t > 0$ . If $2|r|<s$ and $2|r|<t$,  then we can find Jones surfaces which are spanning surfaces of $K$ \cite{LV} for each Jones slope. For each Jones surface the numbers of sheets is 1, which clearly divides the Jones period $p=1$ of $K$.

If $|r| > s$ or $|r| > t$, the Jones period is equal to $\displaystyle{p= \frac{-2+s+t}{2}}$, the Jones slopes are given by $js_K=\{\displaystyle{2\left(\frac{1-st}{-2+s+t} - r\right)} \} $ and 
$js_K^* = \{-2(s+t)\}$. A Jones surface $S$ with boundary slope equal to $s\in js_K$ has number of sheets the least common multiple of the denominators of two fractions $\displaystyle{-\frac{1-st}{-2+s+t}}$ and $\displaystyle{\frac{t-1}{-2+s+t}}$, each reduced to lowest terms. Since $s, t$ are odd, the resulting reduced fractions all have denominators dividing $\displaystyle{p= \frac{-2+s+t}{2}}$. For details on how the fractions are assigned and their relations to the number of sheets of $S$, see \cite{LV}. A Jones surface with boundary slope equal to $s^*\in js_K^{*}$ is a spanning surface of $K$. In both cases the number of sheets divides the period and hence  the Jones surfaces are characteristic.

 For example, the pretzel knot $P(-1/101, 1/35, 1/31)$ has a Jones slope $js_K= \{1345/8\}$ realized by a Jones surface with 32 sheets. The 32 sheets comes from taking the least common multiple of the denominators of the fractions $271/16$ and $15/32$.  This means that the number of boundary components is 4. The Jones period is 32. This is an interesting example where both $b$ and $|\partial S|$ are not equal to 1. Yet another interesting example comes from this family--the pretzel knot $P(-1/101, 1/61, 1/65)$, which has Jones period $p=62$. It has a Jones slope $js_K=\{4280/31\}$ from a Jones surface of 31 sheets (the corresponding reduced fractions are $991/31$ and $16/31$) which divides the Jones period 62, but is not equal to it. 
\end{example} 

We note that currently there are no examples of knots which admit multiple Jones slopes for either $d_+[J_K(n)]$ or $d_-[J_K(n)]$.  That is, in all the known cases the functions $a_K(n), a_K^{*}(n)$ are both constant.
One may ask the following

\begin{question} Are there knots for which  $a_K(n), a_K^{*}(n)$  are not eventually constant functions? That is, is there a knot $K$ that admits multiple Jones slopes 
 for $d_+[J_K(n)]$ or $d_-[J_K(n)]$? 
\end{question}  

The discussion above and examples also raise the following question. 

\begin{question} \label{lessthanp}
Is it true that for every Jones slope of a knot $K$ we can find a characteristic Jones surface?
\end{question}


\smallskip

\section{Haken sums for Jones surfaces} 
In this section we show that there is a normal surface theory algorithm to decide whether a  given knot satisfies the Strong Slope Conjecture.

Here we will briefly recall a few  facts about normal surfaces.
For more background and terminology on normal surface theory, the reader is referred to \cite{MatveevAlgorithmicBook}, \cite{JacoTo}, or the introduction of  \cite{JacoDecisionProblems}. Our notations here closely follow that of \cite{howie}.

Let $M$ be a 3-manifold with a triangulation  $\mathcal T$ consisting of $t$ tetrahedra. A properly embedded surface $S$ is called {\emph{normal}} if for every tetrahedron $\Delta$, 
the intersection $\Delta\cap S$  consists of triangular or quadrilateral discs each  intersecting each edge of the tetrahedron  in at most one point and away from vertices of $\mathcal T$. There are seven normal isotopy classes of normal discs, four are triangular and three are quadrilateral; these are called {\emph{disc types}}. Thus we have total of $7t$ normal disc types in  $\mathcal T$. Fixing an order of these normal discs, $D_1,\ldots, D_{7t}$,
  $S$ is represented by a unique (up to normal isotopy) $7t$-tuple of non-negative integers ${\bf n}(S)= (y_1, \ldots, y_{7t})$, where $y_i$ is the number of the discs $D_i$ contained in 
 $S$.

 Conversely, given a $7t$-tuple of non-negative integers ${\bf n}$, we can impose constraints  on the $y_i$'s so that it represents a unique up to isotopy normal surface in  $\mathcal T$.  These constraints  are known as normal surface equations.

\begin{define} {\rm Two normal surfaces $S_1, S_2$ are called {\em{compatible}} if they do not contain quadrilateral discs of different types. Given compatible normal surfaces $S_1, S_2$
 one can form their {\emph{Haken sum}} $S_1\oplus S_2$: This 
 is a geometric sum along each arc and loop of $S_1\cap S_2$ and it is uniquely determined by the requirement that the resulting surface $S_1\oplus S_2$ be normal in $\mathcal T$.
   See Figure \ref{sum}.}
 \end{define}

\begin{figure}[ht]
\includegraphics[scale=.5]{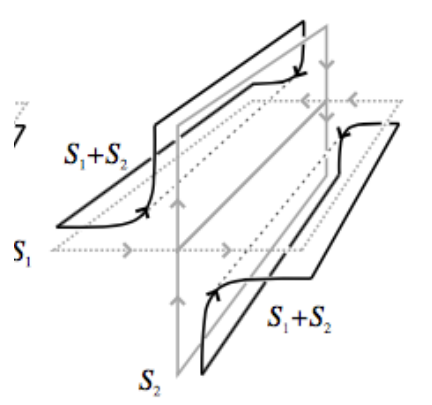}
\hspace{0,3cm} 
 \includegraphics[scale=.5]{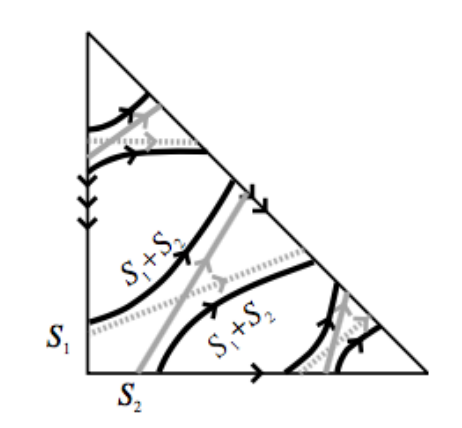}
\caption{The Haken sum of compatible normal surfaces $S_1$ and $S_2$ (left) and the resulting normal curves from $\partial S_1$ and $\partial S_2$ (right).}
\label{sum}
\end{figure}  
 
 If ${\bf n}(S_1)= (y_1, \ldots, y_{7t})$ and  ${\bf n}(S_2)= (y'_1, \ldots, y'_{7t})$, then  $${\bf n}(S_1\oplus S_2)={\bf n}(S_1)+{\bf n}(S_2) = (y_1+y'_1, \ldots, y_{7t}+y'_{7t}),$$
 and  $\chi(S_1\oplus S_2)= \chi(S_1)+\chi(S_2)$.
   
 \begin{define}
 \label{fundamental} 
{\rm  A normal surface $S$ is called \emph{fundamental} if  ${\bf n}(S)$  cannot be
written as a sum of two solutions to the normal surface equations.}
\end{define} 

There are only finitely many fundamental surfaces and there is an algorithm to find all of them. Furthermore, 
all normal surfaces 
can be written as a finite sum of fundamental surfaces \cite{Haken}.

\begin{theorem} \label{mainjones } Given a knot $K$ with known sets $js_K, js^{*}_K, jx_K, jx^{*}_K$ and known Jones period $p$, there is a normal surface theory algorithm that decides whether $K$ satisfies the Strong Slope Conjecture.
\end{theorem}

\begin{proof}  There is an algorithm to determine whether  $M_K=S^3\setminus n(K)$ is a solid torus and thus  if $K$ is the unknot \cite{Haken, JacoTo}.
If $K$ is the unknot then the Strong Slope Conjecture is known and we are done.

If $K$ is not the unknot then we can obtain a triangulation  $\mathcal T_1$ of the complement $M_K$ together with a meridian of $M_K$ that is expressed as a path that follows edges of  $\mathcal T_1$ on $\partial M_K$. A process for getting this triangulation is given in \cite[Lemma 7.2]{HLP}.
Apply the algorithm of Jaco and Rubinstein \cite[Proposition 5.15 and Theorem 5.20]{JR} to convert $\mathcal T_1$ to a triangulation $\mathcal T$ that has a single vertex (a  one-vertex triangulation) and  contains  no normal embedded $2$-spheres.  The algorithm ensures that the only vertex of the triangulation lies on $\partial M_K$.
Then we can apply the process known as ``layering" a triangulation to alter the edges on $\partial M_K$ till the meridian becomes a single edge in the triangulation
(see  \cite{JacoDecisionProblems}). We will continue to denote this last triangulation by $\mathcal T$ and we will use $\mu$ to denote the single edge corresponding to the meridian of $K$.

For notational simplicity we will work with $js_K$ and  $jx_K$ as the argument for $js^{*}_K$  and $ jx^{*}_K$ is completely analogous. 
Fix a Jones slope  $a/b\in js_K$, with $b>0$ and $\gcd(a, b)=1$, and suppose that we have Jones surfaces corresponding to it.
Let $S$ be such a surface   with
$\displaystyle{ \beta:=\frac{\chi(S)}{{\abs{\partial S} b}} \in jx_K.}$  By Proposition \ref{divide},  $\abs{\partial S} b$ divides $2p^2 \chi(S)$, where $p$ is the Jones period of $K$.
Thus \begin{equation}
2p^2 \chi(S)-\lambda  \abs{\partial S} b=0 \ \  {\rm where} \ \    \lambda=2p^2\beta  \in \ZZ.
\label{JS}
\end{equation}

\smallskip

 \begin{lemma}\label{fundamental} Suppose that $S$ is a Jones surface with boundary slope $s:=a/b\in js_K$, where $K$ has Jones period $p$, and  with
$\beta=\displaystyle{ \frac{\chi(S)}{{\abs{\partial S} b}} \in jx_K}$ and $\lambda$ as defined in \eqref{JS}. 
Then, exactly one of the following is true:
\begin{enumerate}
\item There is a Jones surface $\Sigma$   with boundary slope $a/b\in js_K$ and with
\begin{equation}
 \frac{\chi(\Sigma)}{{\abs{\partial \Sigma} b}}=\beta= \frac{\lambda}{2p^2},
 \label{JS1}
 \end{equation}
that is also a  normal fundamental surface with respect to  $\mathcal T$.

\item There is a nonempty set ${\mathcal EZ'}_s$ of
essential surfaces that are normal fundamental surfaces with respect to  $\mathcal T$, 
have boundary slope $s$ and for every $\Sigma\in{\mathcal EZ'}_s$ we have

$$2p^2\chi(\Sigma)- \lambda \abs{\partial {\Sigma}}b\neq 0 \ \  {\rm and} \ \  \abs{\partial {\Sigma}}\leq \abs{\partial {S}}.  $$

\end{enumerate}
\end{lemma}

 \begin{proof} 
Let $S$ be a Jones surface as above. Any essential surface in $M_K$ may be isotoped to a normal surface with respect to above fixed  $\mathcal T$.
 Moreover, this normal surface $S$ may be be taken to be minimal in the sense of \cite[Definition 4.1.6]{MatveevAlgorithmicBook}: This means that the surface minimizes the
number of intersections with the 1-skeleton  $\mathcal T^1$ of $\mathcal T$ in the (normal) isotopy class of the surface.  The significance of this  minimality condition is the following: By
\cite[Corollary 4.1.37]{MatveevAlgorithmicBook}, applied to $ (M_K, \mu)$, if $S$ can be written as a Haken sum of non empty normal surfaces then each of these summands  
is essential in $M_K$. See also \cite[Theorem 5.1]{JacoDecisionProblems}.

Suppose that $S$ is not fundamental. Then $S$ can be represented as a \emph{Haken sum} 
\begin{equation}
S = \Sigma_1 \oplus \hdots \oplus \Sigma_n \oplus F_1 \oplus \hdots \oplus F_k, \label{Haken}
\end{equation}
where each $\Sigma_i$ is a fundamental normal surface with boundary, and each $F_i$ is a closed fundamental normal surface. By Jaco and Sedgwick
 \cite[Proposition 3.7 and Corollary 3.8]{JacoDecisionProblems} each $\Sigma_i$ has the same  boundary slope as $S$.  As said earlier we have

\begin{equation}
\chi(S)= \chi(\Sigma_1) + \hdots +\chi(\Sigma_n) + \chi( F_1) + \hdots + \chi(F_k).
\label{ki}
\end{equation}

Recall that the number of sheets of  a surface $S $, that is properly embedded in $M_K$,  is the number of  intersections of $\partial S$ with the edge $\mu$. We also recall that the boundary of a Haken sum $S_1\oplus S_2$  is obtained  by resolving  the double points in $\partial S_1\cap \partial S_2$ so that the resulting curves are still normal.
In particular, the homology class of $\partial(S_1\oplus S_2)$ is the sum of the homology classes of $\partial S_1$ and $\partial S_2$ in $H_1(\partial M_K)$. This implies that the number of intersections of
 $\partial(S_1\oplus S_2)$  with $\mu$ is the sum of the numbers of intersection of   $\partial S_1$ and $\partial S_2$   with $\mu$. Thus by (\ref{Haken}) we obtain

\begin{equation}
\abs{\partial S} b= \abs{\partial {\Sigma_1}}b + \hdots +\abs{\partial {\Sigma_n}}b.
\label{bdry}
\end{equation}
As said above, \cite[Corollary 4.1.37]{MatveevAlgorithmicBook}  shows that $\Sigma_i$ must be essential, for all $i=1,\ldots, n$.

If for some $i$ we have $2p^2\chi(\Sigma_i)- \lambda \abs{\partial {\Sigma_i}}b=0,$ then $\Sigma:=\Sigma_i$ is a Jones surface as claimed in (1) in the statement of the lemma. Otherwise we 
have \begin{equation}
2p^2\chi(\Sigma_i)- \lambda \abs{\partial {\Sigma_i}}b\neq 0,
\label{positive}
\end{equation}
for all $1\leq i\leq n$. It follows immediately by equation (\ref{bdry}) that  $ \abs{\partial {\Sigma_i}}\leq  \abs{\partial {S}}$ and hence option (2) is satisfied. \end{proof}

\vskip 0.03in

To continue suppose that there exist Jones surfaces  $S$, with boundary slope $s:=a/b$, with $\displaystyle{\beta =\frac{\chi(S)}{|\partial S|b}}$, and $\lambda$ defined as in (\ref{JS}),
but there are no such surfaces  that are fundamental with respect to $\mathcal T$. Then, by Lemma  \ref{fundamental},
 we have a set ${\mathcal EZ'}_s\neq \emptyset$ of properly embedded essential surfaces in $M_K$ such that for every $\Sigma_i \in {\mathcal EZ'}_s$ we have:
\begin{itemize}
\item $\Sigma_i$ has boundary slope $s$ and  is  a  normal fundamental surface with respect to  $\mathcal T$; and

\item we have $2p^2\chi(\Sigma_i)- \lambda \abs{\partial {\Sigma_i}}b\neq 0.$
\end{itemize}

By the proof of Lemma  \ref{fundamental}, a Jones surface $S$ as above  is a Haken sum of essential fundamental  surfaces 

\begin{equation}
 S= (\oplus_{i} n_i \Sigma_i) \oplus \left( \oplus_{j}m_j F_j \right),
 \label{normalsum}
 \end{equation}
where $\Sigma_i\in {\mathcal  EF}_s$,  the $F_j$'s are closed surfaces and $n_i, m_j\geq 0$ are integers.
We have
$$\chi(S)=\sum_{i}\chi(\Sigma_i) n_i  + \sum_{j} \chi(F_j) m_j,$$
and

$$\abs{\partial S} b= \sum_{i}\abs{\partial {\Sigma_i}}b   n_i.$$

Multiplying the first equation by $2p^2$, the second by $\lambda$  and subtracting we obtain
\begin{equation}
\sum_{i}x(\Sigma_i) n_i  + 2p^2\sum_{j} \chi(F_j) m_j=0
\label{homozero}
\end{equation}

where $$x(\Sigma_i):=2p^2\chi(\Sigma_i)- \lambda \abs{\partial {\Sigma_i}}b\neq 0.$$

Thus the vector  ${\bf n:}=(n_1, \ldots,  m_1,\ldots )$ corresponds  to a solution of the homogeneous equation 
(\ref{homozero}), with non-negative integral entries.  We recall that a solution vector  ${\bf n}$ with non-negative integer entries, for equation (\ref{homozero}), is called fundamental if it cannot be written as a non-trivial sum of solution vectors with non-negative integer entries. 
For any system of linear homogeneous equations, there is a finite number of fundamental solutions that can be found algorithmically, 
and every solution is linear combination of fundamental ones (see, for example, \cite[Theorem 3.2.8]{MatveevAlgorithmicBook}).

\begin{lemma} \label{fundamental2} Suppose that there is a Jones surface $S$
corresponding to boundary slope $s$ which satisfies equation (\ref{JS}). Suppose moreover that there are no normal fundamental surfaces with respect to $\mathcal T$ that
are Jones surfaces satisfying (\ref{JS1}). Then,
there is a Jones surface  $\Sigma',$  with boundary slope $s$ and  $\displaystyle{\frac{\chi(\Sigma')}{{\abs{\partial \Sigma'} b}}=\beta= \frac{\lambda}{2p^2}}$, such that 
$$\Sigma'= \oplus_{i} k_i \Sigma'_i \oplus \left( \oplus_{j}l_j F_j \right),$$
where ${\bf k}=(k_1,\ldots, l_1, \ldots, )$ is a fundamental solution of equation (\ref{homozero}).
\end{lemma}
\begin{proof} By assumption we have a Jones surface $S$ of the form shown  in equation (\ref{normalsum}) corresponding to a solution ${\bf n}$ with non-negative integer entries of equation (\ref{homozero}).
If ${\bf n}$  is not fundamental, then ${\bf n}= {\bf k}+{\bf m} $
where ${\bf k}$ is fundamental and ${\bf m} $ a non-trivial solution with non-negative integer entries of equation (\ref{homozero}),  corresponding to normal surfaces $\Sigma'$ and $\Sigma''$ via equation
(\ref{normalsum}). We have $S=\Sigma'\oplus \Sigma''$.
In order for $\Sigma'$ to be a Jones surface it is enough to see that $\Sigma'$ is essential. But this follows by \cite[Corollary 4.1.37]{MatveevAlgorithmicBook} as noted earlier.
\end{proof}

Now we are ready to present our algorithm and  finish the proof of the theorem: Given a knot $K$ with known sets $js_K, js^{*}_K, jx_K,jx^{*}_K$, and known Jones period $p$, to check whether it satisfies the Strong Slope Conjecture we need to
check that the elements in $js_K\cup   js^{*}_K$  are boundary slopes and to find Jones surfaces for all these slopes.
To use Lemma \ref{fundamental},   we need to know the 
 fundamental
normal surfaces with respect to the triangulation $\mathcal T$ fixed in the beginning of the proof. There are  finitely many fundamental surfaces in $M_K$ and there  is an algorithm to  find them \cite{Haken}.

\subsection*{Algorithm for finding Jones surfaces.} Let  $\mathcal{Z}=\{ Z_1, \ldots, Z_k\}$ denote the list of all
fundamental surfaces. There is an algorithm to compute $\chi(Z)$ for all surfaces $Z \in \mathcal{Z},$ and  to compute the boundary slopes of the ones with boundary \cite{JacoTo}. 
Let $$\mathcal{A}=\displaystyle{\{  {a_1}/{b_1}, \ldots, {a_s}/{b_s}\}}$$ denote the list of distinct finite boundary slopes of the surfaces in $\mathcal{Z}$, where 
$(a_i, b_i)=1$ and $b_i>0$. Now proceed as follows:
\begin{enumerate}
\item  Check whether  $js_K \subset \mathcal{A}$  and    $js^{*}_K \subset \mathcal{A}$. If one of the two inclusions fails then $K$ does not satisfy the Slope Conjecture.

\item If $ \mathcal{Z}$  contains no closed surfaces move to the next step. If we have closed surfaces we need to find any incompressible ones among them.
 There is an algorithm that decides whether a given 2-sided surface is incompressible and boundary incompressible if the surface has boundary. See \cite[Theorem 4.1.15]{MatveevAlgorithmicBook}, \cite[Theorem 4.1.19]{MatveevAlgorithmicBook}, or
 \cite[Algorithm 3]{Burtonb}.
Apply the algorithm to each closed surface in $ \mathcal{Z}$  to decide whether  they are
incompressible. Let ${\mathcal C}\subset \mathcal{Z}$  denote the set of incompressible  surfaces found, that have genus bigger than one.

\item For every $s:={a}/{b} \in js_K\subset  \mathcal{A}$ consider the set ${\mathcal{Z}}_s\subset \mathcal{Z}$ that have boundary slope  ${a}/{b}$. 
By \cite{JacoDecisionProblems} we know that ${\mathcal{Z}}_s\neq \emptyset$. Decide whether ${\mathcal{Z}}_s$  contains essential surfaces and find them.
 Note that the surfaces  in  $\mathcal{Z}_s$ may not be 2-sided. To decide that an 1-sided surface is essential one applies the incompressibility and $\partial$-incompressibility algorithm to the double of the surface.
 Let $\mathcal{EZ}_s$ denote the set of essential surfaces found. If $\mathcal{EZ}_s=\emptyset$ then  $K$ fails the conjecture.

\item 
 For every $\lambda \in {2p^2 jx_K}$ and every $\Sigma\in {\mathcal{EZ}}_s$
 calculate the quantity $$x(\Sigma):= 2p^2\chi(\Sigma) -\lambda b \abs{\partial \Sigma}.$$
 Suppose that there is $\Sigma\in {\mathcal{EZ}}_s$ with  $x(\Sigma)=0$. Then  any such $\Sigma$ is a Jones surface corresponding to $s$.

 \item
Suppose ${{\mathcal EZ'}_s}:= \{ \Sigma_1, \ldots, \Sigma_r \} \neq\emptyset$ and that we have $x(\Sigma)\neq 0$,   for all $\Sigma\in {\mathcal{EZ}'}_s$.    Then consider equation (\ref{homozero}) 
 $$x(\Sigma_1) n_1+\ldots + x(\Sigma_r) n_r+2p^2 \chi(C_1) m_1+\ldots +  2p^2 \chi(C_t) m_t=0,$$
 where $C_i$ runs over all the surfaces in  ${\mathcal C}$. Find and enumerate all the fundamental solutions $\Sigma'$ of the equation as in Lemma \ref{fundamental2}. Among these solutions pick the {\emph {admissible}} ones: That is 
 solutions for which, for any incompatible  pair of surfaces in   ${\mathcal C}\cup {\mathcal EZ}'_s$, at most one of the corresponding entries in the solution
 should be non-zero. Hence pairs of non-zero numbers correspond to pairs of compatible surfaces.
 Every admissible fundamental solution represents a normal surface. By Lemma \ref{fundamental2}, we need only to check if one of these surfaces is essential.  If a surface in this set is essential, then it is a Jones surface, otherwise, $K$ fails the Strong Slope Conjecture.

  \item For every   $a/b\in js_K\subset  \mathcal{A}$  repeat steps (3)-(5) above and run the analogous process for the Jones slopes in $ js_K^{*}$.
 \end{enumerate}
\end{proof}

The next Corollary gives conditions where Jones surfaces can be chosen to be fundamental.
\begin{corollary}\label{funda} Suppose that $S$ is a Jones surface with boundary slope $s:=a/b\in  js_K\cup js_K^{*}$. Suppose moreover that $S$ is a spanning surface of $K$ (i.e. ${{\abs{\partial S} b}} =1$)
that has maximal Euler characteristic over all spanning surfaces of $K$ with boundary slope $s$. Then there is a Jones surface $\Sigma_1$ corresponding to $s$ 
that is also a normal fundamental surface with respect to  $\mathcal T$.

\end{corollary}

\begin{proof} Consider a {Haken sum} decomposition of $S$ as in equation (\ref{Haken}) in the proof of Lemma \ref{fundamental}. Since
 $b \abs{\partial {\Sigma_i}}\leq  b\abs{\partial {S}}=1$, we obtain that  $b \abs{\partial {\Sigma_i}}=1$ and by equation (\ref{bdry}) we have $n=1$. Thus
\begin{equation}
S = \Sigma_1 \oplus F_1 \oplus \hdots \oplus F_k, 
\end{equation}

\noindent where $\Sigma_1 $ is a fundamental essential spanning surface of slope $s$ and each $F_i$ is a closed fundamental, incompressible, normal surface with
$ \chi( F_i)\leq 0$. If we have $ \chi( F_i)\neq 0$, for some $1\leq i\leq k$, then $\chi(S)<  \chi(\Sigma_1) $. Since the latter inequality contradicts the assumption that 
$S$ has  maximal Euler characteristic over all spanning surfaces of $K$ with boundary slope $s$, it follows that $ \chi( F_i)=0$ for all $1\leq i\leq k$.
Thus $\chi(S)= \chi(\Sigma_1)$ and, since $b \abs{\partial {\Sigma_1}}=1$ and $\Sigma_1 $ has boundary slope $s$, it follows that $\Sigma_1 $ is a Jones surface.
\end{proof}

Corollary \ref{funda} applies to positive closed braids: Let $B_n$ denote the braid group on $n$ strings, and let $\sigma_1, \cdots,
\sigma_{n-1}$ be the elementary braid generators.  Let $D_b$ denote the closed braid diagram
 obtained from the braid
$b=\sigma_{i_1}^{r_1}\sigma_{i_2}^{r_2} \cdots \sigma_{i_k}^{r_k}$. If $r_j > 0$ for all $j$, the positive braid diagram $D_b$ will be $A$--adequate.  Let $K$ denote the knot represented by $D_b$.
By \cite[Example 9]{FKP}, 
$ js^{*}_K=\{0\}$ and an essential surface $S_A$ that realizes this Jones slope is a fiber in  $S^3 \setminus K$ (thus an orientable Seifert surface of maximal Euler characteristic).
By \cite[Theorem 3.9]{Effie-Anh-slope} and its proof, $S_A$ is a Jones surface of $K$ corresponding to slope zero. Thus the hypotheses of 
Corollary \ref{funda} are satisfied and in this case we can find a Jones surface that is  fundamental with respect to  $\mathcal T$.

At this writing we do not know if there are examples of knots with Jones slopes that do not admit Jones surfaces that are  fundamental with respect to  $\mathcal T$.
In other words we do not know if there are examples  of Jones slopes where step (5) of above given algorithm is needed in order to find the corresponding Jones surfaces.
We ask the following.

\begin{question} \label{five}Is there a knot $K$ that satisfies the Strong Slope Conjecture  and such that there  a Jones slope $s\in  js_K\cup js_K^{*}$ for which we cannot find a Jones surface
that is a normal fundamental surface with respect to  $\mathcal T$?

\end{question}

\begin{remark} Suppose that Question \ref{lessthanp} has an affirmative answer: That is for every Jones slope $s:=a/b$ there is a Jones surface $S$, with 
$\chi(S)= \beta \abs{\partial S} b$, for some $\beta$ such that $\beta$ or $-\beta \in jx_K\cup jx_K^{*}$, that is characteristic (i.e. $\abs{\partial S} b$ divides  the period $p$).
Thus $\abs{\partial S} b\leq p$ and  
\begin{equation}
-\chi(S)+ \abs{\partial S} b\leq  (1-\beta)p.
\label{bounded}
\end{equation}
Now \cite[Theorem 6.3.17]{MatveevAlgorithmicBook}, applied to $(M_K, \mu)$ implies that there are finitely many essential surfaces in $M_K$ that satisfy (\ref{bounded}) and they can be found algorithmically.
Using  this observation, one can see that a positive answer to Question \ref{lessthanp} will lead to an alternative algorithm for finding Jones surfaces than the one outlined above.
\end{remark}


\bibliographystyle{plain} \bibliography{biblio}
\end{document}